\documentclass[preprint,11pt]{article}

\usepackage{amssymb}
\usepackage{amsfonts,amsthm, amsmath}
\usepackage{epsfig}
\usepackage{makeidx}
\usepackage{graphicx,epstopdf}
\usepackage[title]{appendix}
\usepackage{float}
\usepackage{comment}
\usepackage{enumerate}
\usepackage{mathtools}
\usepackage{subfig}
\usepackage{hyphenat}
\DeclareGraphicsExtensions{.ps,.png,.jpg}

\makeindex
\usepackage{color}

\newcommand{\ncom}{\newcommand}
\ncom{\ul}{\underline}
\ncom{\beq}{\begin{equation}}
\ncom{\eeq}{\end{equation}}
\ncom{\bea}{\begin{eqnarray*}}
\ncom{\eea}{\end{eqnarray*}}
\ncom{\beqa}{\begin{eqnarray}}
\ncom{\eeqa}{\end{eqnarray}}
\ncom{\nno}{\nonumber}
\ncom{\non}{\nonumber}
\ncom{\ds}{\displaystyle}
\ncom{\half}{\frac{1}{2}}
\ncom{\mbx}{\makebox{.25cm}}
\ncom{\hs}{\mbox{\hspace{.25cm}}}
\ncom{\rar}{\rightarrow}
\ncom{\Rar}{\Rightarrow}
\ncom{\noin}{\noindent}
\ncom{\bc}{\begin{center}}
\ncom{\ec}{\end{center}}
\ncom{\sz}{\scriptsize}
\ncom{\rf}{\ref}
\ncom{\s}{\sqrt{2}}
\ncom{\sgm}{\sigma}
\ncom{\Sgm}{\Sigma}
\ncom{\psgm}{\sigma^{\prime}}
\ncom{\dt}{\delta}
\ncom{\Dt}{\Delta}
\ncom{\lmd}{\lambda}
\ncom{\Lmd}{\Lambda}
\ncom{\Th}{\Theta}
\ncom{\e}{\eta}
\ncom{\eps}{\epsilon}
\ncom{\pcc}{\stackrel{P}{>}}
\ncom{\lp}{\stackrel{L_{p}}{>}}
\ncom{\dist}{{\rm\,dist}}
\ncom{\sspan}{{\rm\,span}}
\ncom{\re}{{\rm Re\,}}
\ncom{\im}{{\rm Im\,}}
\ncom{\sgn}{{\rm sgn\,}}
\ncom{\ba}{\begin{array}}
\ncom{\ea}{\end{array}}
\ncom{\hone}{\mbox{\hspace{1em}}}
\ncom{\htwo}{\mbox{\hspace{2em}}}
\ncom{\hthree}{\mbox{\hspace{3em}}}
\ncom{\hfour}{\mbox{\hspace{4em}}}
\ncom{\vone}{\vskip 2ex}
\ncom{\vtwo}{\vskip 4ex}
\ncom{\vonee}{\vskip 1.5ex}
\ncom{\vthree}{\vskip 6ex}
\ncom{\vfour}{\vspace*{8ex}}
\ncom{\norm}{\|\;\;\|}
\ncom{\integ}[4]{\int_{#1}^{#2}\,{#3}\,d{#4}}
\ncom{\vspan}[1]{{{\rm\,span}\{ #1 \}}}
\ncom{\dm}[1]{ {\displaystyle{#1} } }
\ncom{\ri}[1]{{#1} \index{#1}}

\newtheorem{theorem}{\bf Theorem}[section]
\newtheorem{remark}{\bf Remark}[section]
\newtheorem{proposition}{Proposition}[section]
\newtheorem{lemma}{Lemma}[section]

\newtheorem{definition}{Definition}[section]
\newtheoremstyle
{remarkstyle}
{}
{11pt}
{}
{}
{\bfseries}
{:}
{     }
{\thmname{#1} \thmnumber{#2} }

\theoremstyle{remarkstyle}

\begin{document}

\begin{center}
{\Large \bf Humbert Generalized Fractional Differenced ARMA Processes}
\end{center}
\vone
\vone
\begin{center}
{Niharika Bhootna}$^{\textrm{a}}$, {Monika Singh Dhull}$^{\textrm{a}}$, {Arun Kumar}$^{\textrm{a}}$, {Nikolai Leonenko}$^{\textrm{b*}}$
        $$\begin{tabular}{l}
            $^{\textrm{a}}$ \emph{Department of Mathematics, Indian Institute of Technology Ropar, Rupnagar, Punjab - 140001, India}\\
            $^{\textrm{b}}$ \emph{Cardiff School of Mathematics, Cardiff University, Senghennydd Road, Cardiff CF24 4AG, UK}\\
            $^{\textrm{*}}$ \emph{Correspondence: LeonenkoN@cardiff.ac.uk}
      \end{tabular}$$
  \end{center}

\begin{abstract}
%% Text of abstract
In this article, we use the generating functions of the Humbert polynomials to define two types of Humbert generalized fractional differenced ARMA processes. We present stationarity and invertibility conditions for the introduced models. The singularities for the spectral densities of the introduced models are obtained. In particular, Pincherle ARMA, Horadam ARMA and Horadam-Pethe ARMA processes are studied.
\end{abstract}

\vone
\vone
\vone

\section{Introduction}\label{sec1}

The study of fractionally differenced time series by Granger and Joyeux (1980) \cite{Granger1980} and Hosking in 1981 \cite{Hosking1981} provided an impetuous to a new research direction in time series modelling. The fractionally differenced time series called the autoregressive fractionally integrated moving average (ARFIMA) model generalizes the autoregressive (AR), moving average (MA) and autoregressive moving average (ARMA) models defined respectively by Yule (1926) \cite{Yule1926}, Slutsky (1937) \cite{Slutsky1937} and Wold (1938) \cite{Wold1938}. 
Also, the ARFIMA model is an extension of the autoregressive integrated moving average (ARIMA) process defined by Box and Jenkins (1976) \cite{Box1976} to model non-stationary time series by assuming the order of differencing $\nu\in \mathbb{R}$. The fractionally differenced time series is useful to model the data exhibiting long range dependence (LRD). The data exhibiting LRD behaviours or long memory have high correlation after a significant lag. Anh et al. proposed some continuous time stochastic processes with seasonal long range dependence and these kind of long memory processes have spectral pole at non-zero frequencies \cite{Anh2004_2}. In subsequent years, Andel (1986); Gray, Zhang and Woodward (1989, 1994) introduced the concept of Gegenbauer ARMA (GARMA) process. GARMA process also possess seasonal long range dependence \cite{Castro2021}. The study on usefulness of Gegenbauer stochastic process is done by Dissanayake et al. \cite{Dissanayake2018}. The limit theorems for stationary Gaussian processes and their non-linear transformations with covariance function 
\begin{align*}
&\rho(h) \simeq \sum\limits_{k=1}^{r}A_{k}\cos(h\omega _{k})h^{-\alpha_{k}}, \sum\limits_{k=1}^{r}A_{k}=1,
\end{align*} 
where $A_{k}\geq 0,\alpha_{k}>0,\omega _{k}\in \lbrack 0,\pi ),k=1,\cdots,r$ have been considered in \cite{Ivanov2013}. For seasonal long memory process ${X_t}$, the autocorrelation function for lag $h$ denoted by $\rho(h)$ behaves asymptotically as $\rho(h) \simeq \cos(h\omega_0)h^{-\alpha}$ as $h\rightarrow  \infty$ for some positive $\alpha \in (0, 1)$ and $\omega_0\in(0,\pi)$ (see \cite{Chung1996}). In literature, many tempered distributions and processes are studied using the exponential tempering in the original distribution or process see e.g. and references therein \cite{ Kumar2014, Rosinski2007, Sabzikar2019, Torricelli2022, Zheng2015, Grabchak2016,  Baeume2010}. The fractionally integrated process with seasonal components are studied and maximum likelihood estimation is done by Reisen et al. \cite{Reisen}. The parametric spectral density with power-law behaviour about a fractional pole at the unknown frequency $\omega$ is analysed and Gaussian estimates and limiting distributional behavior of estimate is studied by Giraitis and Hidalgo \cite{Giraitis}. The autoregressive tempered fractionally integrated moving average (ARTFIMA) process is obtained by using exponential tempering in the original ARFIMA process \cite{Sabzikar2019}. The ARTFIMA process is semi LRD and has a summable autocovariance function. In ARIMA process the fractional differencing operator $(1-B)^{\nu},\;\lvert\nu\rvert<1$ is considered instead of $(1-B)$, where $B$ is the shift operator. In defining ARTFIMA model the tempered fractional differencing operator $(1-e^{-\lambda}B)^{\nu}$ is used where $\lambda>0$ is the tempering parameter. The Gegenbauer process uses $(1-2uB+B^2)^{\nu},\; \lvert u\rvert \leq 1,\;\lvert \nu\rvert<\frac{1}{2}$ as a difference operator, which can be written in terms of Gegenbauer polynomials.

\noindent In this article, we study Humbert polynomials based time series models. The Gegenbauer and Pincherle polynomials are the particular cases of Humbert polynomials. The Gegenbauer polynomials based time series model, namely GARMA process, is already studied and has been applied in several real world applications emanating from different areas. We introduce and study two types of Humbert autoregressive fractionally integrated moving average (HARMA) models which are defined by considering Humbert polynomials and obtain the spectral density, stationarity and invertibility conditions of the process. In particular, Pincherle ARMA, Horadam ARMA and Horadam-Pethe ARMA processes are studied. These new class of time series models generalizes the existing models like ARMA, ARIMA, ARFIMA, ARTFIMA and GARMA in several directions.

\noindent The rest of the paper is organized as follows. In Section \ref{sec2}, we introduce the Type 1 HARMA $(p,\nu,u,q)$ process, where $p$ and $q$ are autoregressive and moving average lags respectively and $\nu$ is differencing parameter. This section includes the study of stationarity property and spectral density of the introduced process. Also section \ref{sec2} includes the study of particular case of Type 1 HARMA$(p,\nu,u,q)$ process by taking $m=3$, which is Pincherle ARMA $(p,\nu,u,q)$ process. Moreover, the spectral density of the Pincherle ARMA $(p,\nu,u,q)$ process is obtained and it is shown that for $u=0$ the model exhibits seasonal long memory property. The Section \ref{sec3} deals with the Type 2 HARMA process $(p,\nu,u,q)$. In this section, the particular cases namely Horadam ARMA process and Horadam-Pethe ARMA process are discussed. The last section concludes.

\section{Type 1 HARMA$(p, \nu, u,q)$ Process}\label{sec2}
\setcounter{equation}{0}
In this section, we introduce a new time series model namely type 1 HARMA$(p, \nu, u,q)$ process with the help of Humbert polynomials which we call hereafter type 1 Humbert polynomials. For Humbert polynomials and related properties see e.g. \cite{Humbert1920, Gould1965, Milovanovic1987}. A detailed discussion on special functions including Humbert polynomials is given in \cite{Srivastava1984, Gould1965}.
\begin{definition}[Type 1 Humbert polynomials]\label{def_2}
The Humbert polynomials of type 1 $\{\Pi_{n,m}^\nu\}_{n=0}^{\infty}$ are defined in terms of generating function as
\begin{align}\label{hum}
(1-mut+t^m)^{-\nu}=\sum_{n=0}^{\infty}\Pi_{n,m}^\nu(u) t^n, m\in \mathbb{N}, \lvert t\rvert<1, \lvert u\rvert\leq 1\; {\rm and}\; \lvert\nu\rvert<\frac{1}{2}. 
\end{align}    
\end{definition}

\noindent For the table of main special cases of \eqref{hum}, including Gegenbauer, Legendre,
Tchebysheff, Pincherle, Kinney polynomials, see Gould (1965) \cite{Gould1965}. 
\noindent In above definition, polynomial $\Pi_{n,m}^\nu(u)$ is explicitly can be written as follows \cite{Humbert1920}:
\begin{align*}
\Pi_{n,m}^\nu(u)=\sum_{k=0}^{\lfloor\frac{n}{m}\rfloor}\dfrac{(-mu)^{n-mk}}{\Gamma((1-\nu-n)+(m-1)k)(n-mk)!k!}, \text{ where } \left\lfloor\frac{n}{m}\right\rfloor \text{ is floor function}.
\end{align*}
\noindent The hypergeometric representation of $\Pi_{n,m}^\nu(u)$ is given as follows:
\begin{align*}
\Pi_{n,m}^\nu(u)=\dfrac{(\nu)_n(mu)^n}{n!} {}_mF_{m-1}&\Bigg[\dfrac{-n}{m},\dfrac{-n+1}{m},\hdots,\dfrac{-n-1+m}{n};\\
&\dfrac{-\nu-n+1}{m-1},\dfrac{-\nu-n+2}{m-2},\hdots,\dfrac{-\nu-n+m-1}{m-1};\dfrac{1}{(m-1)^{m-1}u^m}\Bigg].
\end{align*}

\noindent For more properties and results on hypergeometric functions see Srivastava and Manocha (1984) \cite{Srivastava1984}. The type 1 Humbert polynomial satisfies the following recurrence relation
\begin{align*}
    (n+1)\Pi_{n+1,m}^\nu(u)-mu(n+\nu)\Pi_{n,m}^\nu(u)-(n+m\nu-m+1)\Pi_{n-m+1,m}^\nu(u)=0.
\end{align*}
For $m=2$ the Humbert polynomials reduces to Gegenbauer polynomials generally denoted as $\{C_n^\nu(u)\}_{n=0}^{\infty}$ and for $m=3$ the polynomials reduce to Pincherle polynomials $\{P_n^\nu(u)\}_{n=0}^{\infty}$, see Pincherle (1891) \cite{Pincherle1891}. The generating function of Pincherle polynomials have the following form 
\begin{align*}
(1-3ut+t^3)^{-\nu}=\sum_{n=0}^{\infty}P_{n}^\nu(u) t^n,
\end{align*}
\noindent where $P_{n}^\nu(u)$ has the following representation in terms of hypergeometric function \cite{Pincherle1891}

\begin{align*}
P_n^{\nu}(u)=\frac{(\nu)_n (3x)^n}{n!} {}_3F_2\Bigg[\frac{-n}{3},\frac{-n+1}{3},\frac{-n+2}{3};\frac{-n-\nu+1}{2},\frac{-n-\nu+2}{2};\frac{-1}{4x^3}\Bigg],
\end{align*}
where ${}_3F_2(a_1, a_2, a_3; b_1, b_2;x) = \displaystyle\sum_{k=0}^{\infty}\frac{(a_1)_k(a_2)_k(a_3)_k}{(b_1)_k(b_2)_k}\frac{x^k}{k!}$ and $(a_1)_k = \frac{\Gamma(a_1+k)}{\Gamma(a_1)}$ see e.g. \cite{Abramowitz1992}.
\begin{definition}[Type 1 HARMA process]\label{def1}
The type 1 HARMA$(p, \nu, u,q)$ process $X_t$ is defined by
\begin{equation}\label{eq1_}
\Phi(B) (1-muB+ B^m)^\nu X_t= \Theta(B) \epsilon_t,
\end{equation}
where $\epsilon_t$ is Gaussian white noise with variance $\sigma^2$, $B$ is the lag operator, $0\leq u<2/m$, and $\Phi(B)$, $\Theta(B)$ are stationary AR and invertible MA operators respectively, defined as, $$\Phi(B) = 1 - \displaystyle\sum_{j=1}^{p}\phi_jB^{j}, \Theta(B) = 1 + \displaystyle\sum_{j=1}^{q}\theta_jB^{j},  \text{ and } B^j(X_t) = X_{t-j}.$$ 
\end{definition}

\noindent In next result, the stationarity and invertibility conditions of the type 1 HARMA process are given. Also, the Abel's test which will be used in next theorem is stated below as proposition. 
\begin{proposition}[Abel's tests \cite{bartle}]
    If the series $\displaystyle\sum_{n=0}^{\infty} a_n$ is convergent and $\{b_n\}$ is monotone and bounded sequence then series $\displaystyle\sum_{n=0}^{\infty} a_n b_n$ is also convergent.
\end{proposition}

\begin{definition}[Asymptotically equivalent functions \cite{Gamelin}]
The functions $f$ and $g$ are said to be asymptotically equivalent that is, $f(h) \simeq g(h)$  as $h\to\infty$ if $\lim_{h\to \infty} \frac{f(h)}{g(h)} = 1.$
    
\end{definition}
\begin{theorem}\label{theorem3}
Let $\{X_t\}$ be the type 1 \textup{HARMA}$(p,\nu,u,q)$ process defined in \eqref{eq1_} and all roots of $\Phi(B)=0$ and $\Theta(B)=0$ lie outside the unit circle then the \textup{HARMA}$(p,\nu,u,q)$ process is stationary and invertible for $\lvert\nu\rvert<1/2$ and $0\leq u\leq2/m.$ 
\end{theorem}

\begin{proof}
Using \eqref{eq1_}, one can write
\begin{align}
    X_t=\frac{\Theta(B)}{\Phi(B)}(1-muB+B^m)^{-\nu}\epsilon_t, \;\text{where}\;
    \frac{\Theta(B)}{\Phi(B)}=\sum_{j=0}^{\infty}\psi_j B^j.
\end{align}
Further,
\begin{align}
    (1-muB+B^m)^{-\nu}= \sum_{n=0}^{\infty}\frac{(\nu)_n}{n!}\left(muB-B^m\right)^n.
\end{align}

% \begin{align}\label{eq2_}
    % X_t&=\sum_{j=0}^{\infty}\sum_{n=0}^{\infty}\psi_j P_{n}^{\nu}(u)B^{j+n} \epsilon_t
    % \end{align}
% where
% \begin{align}
    % P_n(u)^{\nu}=\frac{(\nu)_n}{n!}(muB-B^m)^n
    % \end{align}
\noindent Then  \eqref{eq1_} can be written as
\begin{align*}
    X_t&=\sum_{j=0}^{\infty}\sum_{n=0}^{\infty}\psi_j\frac{(\nu)_n}{n!}(muB-B^m)^n \epsilon_{t-j-n}\\
    &=\sum_{j=0}^{\infty} \sum_{n=0}^{\infty}\psi_j {\frac{(\nu)_n}{n!}} \sum_{r=0}^{n}(-1)^r \binom{n}{r} (mu)^{n-r}B^{2r} \epsilon_{t-n}\\
    &=\sum_{j=0}^{\infty}\sum_{n=0}^{\infty}\psi_j{\frac{(\nu)_n}{n!}} (mu-1)^{n} \epsilon_{t-n-2j}.
\end{align*}
\noindent The variance of the process $X_t$ is given by
\begin{align*}
    {\rm Var}(X_t) = \sigma^2 \sum_{n=0}^{\infty} \left(\frac{(\nu)_n}{n!}\right)^2 (mu-1)^{2n} 
    = \sigma^2 \sum_{n=0}^{\infty} \left(\frac{\Gamma(\nu+n)}{\Gamma(\nu)\Gamma(n+1)}\right)^2 (mu-1)^{2n}. 
\end{align*}

\noindent Let $a_n=(mu-1)^{2n}$ and $\{b_n\}=\left(\frac{\Gamma(\nu+n)}{\Gamma(\nu)\Gamma(n+1)}\right)^2$, then using Abel's test $\displaystyle\sum_{n=0}^{\infty}a_n$ converges for $0<u<\frac{2}{m}$ and using Stirling's approximation, for large $n$, $b_n\simeq \frac{n^{2\nu-2}}{(\Gamma(\nu))^2}$, which implies that the sequence is bounded for $\nu<\frac{1}{2}$. We can write $b_n= \binom{\nu+n-1}{n}$ and it is known that $n \choose x$ is decreasing for $x\geq \lfloor\frac{n}{2}\rfloor$ this implies that $\{b_n\}$ is decreasing for $\nu\leq 1$. This indicates that the sequence is bounded and monotone for $\nu<1/2$ and the $\mathrm{Var}(X_t)$ converges for the defined range. Similarly to prove the invertibility condition we define the process \eqref{eq1_} as
\begin{align*}
\epsilon_t=\pi(B)X_t,
\end{align*}
where $\pi(B)=\frac{\Phi(B)}{\Theta(B)}(1-muB+B^m)^{\nu}$ and again using the same argument discussed above the $\pi(z)$ will converge for $-\frac{1}{2}<\nu<1$ and $0<u<\frac{2}{m}$. For $u=0$ and $u=\frac{2}{m}$the variance can be defined as follows
\begin{align*}
    {\rm Var}(X_t) &= \sigma^2 \displaystyle \sum_{n=0}^{\infty} \Bigg(\frac{\Gamma(\nu+n)}{\Gamma(\nu)\Gamma(n+1)}\Bigg)^2\\
    & = \sigma^2 \displaystyle \sum_{n=0}^{N} \Bigg(\frac{\Gamma(\nu+n)}{\Gamma(\nu)\Gamma(n+1)}\Bigg)^2 +\displaystyle \sum_{n=N+1}^{\infty} \Bigg(\frac{\Gamma(\nu+n)}{\Gamma(\nu)\Gamma(n+1)}\Bigg)^2.
\end{align*}

\noindent In the above equation, the first summation is finite and the terms inside the second summation behaves like $\frac{n^{2\nu-2}}{\Gamma(\nu)^2}$ for large $n$ and it is bounded for $\nu<\frac{1}{2}$. Hence, the HARMA process is stationary and invertible for $\lvert\nu\rvert<1/2$ and $0\leq u \leq\frac{2}{m}$. 
\end{proof}

\begin{theorem}
For a type 1 \textup{HARMA}$(p, \nu, u, q)$ process defined in \eqref{eq1_}, under the assumptions of theorem \ref{theorem3} the spectral density takes the following form
\begin{align*}
f_x(\omega)&=\dfrac{\sigma^2}{2\pi}\dfrac{\lvert\Theta(z)\rvert^2}{\lvert\Phi(z)\rvert^2}(2+m^2u^2-2mu(\cos(\omega)+\cos((1-m)\omega))+2\cos(m\omega))^{-\nu},
\end{align*}
where $z=e^{-\iota \omega},\; \omega \in (-\pi, \pi)$.
\end{theorem}

\begin{proof}
\noindent Rewrite \eqref{eq1_} as follows
\begin{equation*}
X_t=\Psi(B)\epsilon_t,
\end{equation*}
where $\Psi(B)= \dfrac{\Theta(B)}{\Phi(B)} (1-muB+B^m)^{-\nu}$. Then using the definition of spectral density of linear process, we have
\begin{align}\label{spec_}
f_x(\omega)=\lvert\Psi(z)\rvert^2f_\epsilon(\omega),
\end{align}
where $z=e^{-\iota \omega}$ and $f_\epsilon(\omega)$ is spectral density of the innovation term. The spectral density of the innovation process $\epsilon_t$ is $\sigma^2/2\pi$. Then \eqref{spec_} becomes,
\begin{align*}
f_x(\omega)&=\dfrac{\sigma^2}{2\pi}\lvert\Psi(z)\rvert^2
=\dfrac{\sigma^2}{2\pi}\dfrac{\lvert\Theta(z)\rvert^2}{\lvert\Phi(z)\rvert^2}\left\lvert1-mue^{-\iota\omega}+e^{-m\iota\omega}\right\rvert^{-2\nu}\\[5pt]
&=\dfrac{\sigma^2\lvert\Theta(e^{-\iota\omega})\rvert^2 \left\lvert1-mue^{-\iota\omega}+e^{-m\iota\omega}\right\rvert^{-2\nu}}{2\pi\lvert\Phi(e^{-\iota\omega})\rvert^2}.
\end{align*}

\noindent Here, $\left\lvert1-mue^{-\iota\omega}+e^{-m\iota\omega}\right\rvert^{-2\nu} = (2+m^2u^2-2mu(\cos(\omega)+\cos((1-m)\omega))+2\cos(m\omega))^{-\nu}$ and the spectral density takes the following form
\begin{align}\label{eq5_}
f_x(\omega) &= \dfrac{\sigma^2}{2\pi}\dfrac{\lvert\Theta(z)\rvert^2}{\lvert\Phi(z)\rvert^2}(2+m^2u^2-2mu(\cos(\omega)+\cos((1-m)\omega))+2\cos(m\omega))^{-\nu}.
\end{align}
\end{proof}

\begin{definition}[Singular point \cite{Gamelin}]
The point $\omega=\omega_0$ is said to be singular point of function $f$ if at $\omega=\omega_0$, $f$ fails to be analytic, that is $f(\omega_0) = \infty.$
\end{definition}

%\begin{definition}[Singular spectrum]
%The stationary time series $\{X_t\}$ is said to has %singular spectrum if for some $j$, $f_x(\omega_j)=0$ or %$f_x(\omega_j)\to \infty.$
%\end{definition}

% \begin{definition}[Long memory in terms of autocorrelation function \cite{Beran1994}, p. 42]
% Let $X_t$ be a stationary process with autocorrelation function $\rho(\cdot)$ for which the following holds. There exists a real number $\alpha \in (0, 1)$ and a constant $c_{\rho} > 0$ such that
% $$
% \rho(h) \simeq c_{\rho}h^{-\alpha},\;{\mbox as}\;h\rightarrow \infty.
% $$
% \end{definition}

% \begin{definition}[Long memory in terms of spectral density \cite{Beran1994}, p. 42]
% The stationary time series $\{X_t\}$ is said to have long memory if there exists a real number $\beta \in (0, 1)$ and a constant  $c_f> 0$ such that
% $$
% f_x(\omega)\simeq c_f\lvert\omega\rvert^{-\beta},\; \mbox{as}\; \omega\to 0.$$
% These processes have regularly varying autocorrelation.
% \end{definition}
\noindent Next, the definition of seasonal or cyclic long-memory is given, which is characterized by having a spectral pole at a frequency $\kappa\in\mathbb{R}$ different from 0, see, e.g., \cite{Anh2004_2, Castro2021}.
\begin{definition}[Seasonal long memory]
The stationary time series $\{X_t\}$ is said to have seasonal long memory if there exist $\omega_0\in\mathbb{R}$ and $\alpha\in(0, 1)$ such that
$$
\rho(h)\simeq h^{-\alpha}\cos(h\omega_0),\;\mbox{as}\;h\rightarrow \infty, $$ and $\cos(h\omega_0)\neq 1.$  
\end{definition}

\begin{theorem}\label{theorem4}
Let $\{X_t\}$ be the stationary type 1 \textup{HARMA}$(p, \nu, u, q)$ process and all the assumptions of theorem \ref{theorem3} hold then the spectral density of \textup{HARMA}$(p, \nu, u, q)$ $\{X_t\}$ has singular spectrum
\begin{enumerate}[\itshape(a)]
\item at $u=0$ and $\omega = \frac{4n\pi\pm\pi}{m}$ for $-\frac{m\pm1}{4}<n<\frac{m\mp1}{4};$
\item at $u = \frac{2}{m}(-1)^n\cos(\frac{4n\pi}{m-2})$ and $\omega = \pm \frac{2n\pi}{m-2}$ for $m \neq 2$ and $-\frac{(m-2)}{4}< n<\frac{(m-2)}{4}$;
\item at $\omega = \cos^{-1}(u)$ for $m=2$.
%the spectral density $f_x(\omega) = \dfrac{\sigma^2}{2\pi}\dfrac{\lvert\Theta(z)\rvert^2}{\lvert\Phi(z)\rvert^2}(2+m^2u^2+2\cos(m\omega)-2mu(\cos(\omega)+\cos((m-1)\omega)))^{-\nu}$ has . 
\end{enumerate}
\end{theorem}

\begin{proof}
From \eqref{eq5_}, the spectral density of the process $\{X_t\}$ is
$$f_x(\omega) = \dfrac{\sigma^2}{2\pi}\dfrac{\lvert\Theta(z)\rvert^2}{\lvert\Phi(z)\rvert^2}(2+m^2u^2+2\cos(m\omega)-2mu(\cos(\omega)+\cos((m-1)\omega)))^{-\nu}, \text{ where } z=e^{-\iota \omega}.$$
We consider the denominator and find the zeros as follows,
\begin{align*}
& 2+m^2u^2+2\cos(m\omega)-2mu(\cos(\omega)+\cos((m-1)\omega))\\
& = 2 + 2\cos(m\omega) + [mu - \{\cos(\omega)+\cos((m-1)\omega)\}]^2 - [\cos(\omega)+\cos((m-1)\omega)]^2\\
& = 4\cos^{2}\left[\frac{m\omega}{2}\right] + [mu - \{\cos(\omega)+\cos((m-1)\omega)\}]^2 - 4\cos^{2}\left[\frac{m\omega}{2}\right]\cos^{2}\left[\frac{(m-2)\omega}{2}\right]\\
%& = 4\cos^{2}\left[\frac{m\omega}{2}\right]\sin^{2}\left[\frac{(m-2)\omega}{2}\right] + \left[mu - \{\cos(\omega)+\cos((m-1)\omega)\}\right]^2\\
& = 4\cos^{2}\left[\frac{m\omega}{2}\right]\sin^{2}\left[\frac{(m-2)\omega}{2}\right] + \left[mu - 2\cos\left(\frac{m\omega}{2}\right)\cos\left(\frac{(2-m)\omega}{2}\right)\right]^2.
\end{align*}

\noindent We have the following two cases.
\begin{enumerate}[\itshape(a)] 
\item The first term, $
4\cos^{2}\left[\frac{m\omega}{2}\right]\sin^{2}\left[\frac{(m-2)\omega}{2}\right] \geq 0$ for all $m$ and $-\pi<\omega<\pi$.\\
\begin{align*}
4\cos^{2}\left[\frac{m\omega}{2}\right]\sin^{2}\left[\frac{(m-2)\omega}{2}\right]&=0 \\
\text{ if }\cos^{2}\left(\frac{m\omega}{2}\right) &= 0 \text{ or } \sin^{2}\left(\frac{(m- 2)\omega}{2}\right) = 0 \text{ or both }\\
\Rightarrow \cos\left(\frac{m\omega}{2}\right) &= 0 \text{ for } \omega_1 = (4n\pm1)\frac{\pi}{m}, \text{ for all } m \in \mathbb{ N }\text{ and } \\
&n=0, \pm1, \pm2, \cdots. 
\end{align*}
We find the condition of singularity by solving the second term,
$\left[mu - 2\cos\left(\frac{m\omega}{2}\right)\cos\left(\frac{(2-m)\omega}{2}\right)\right]^2$ at $\omega_1$, which yields $u = 0.$\\
Also, the singular point $\omega_1 \in (-\pi, \pi)$ for $-\frac{m\pm1}{4}<n<\frac{m\mp1}{4}.$\\
Therefore, the type 1 \textup{HARMA}$(p,\nu,u,q)$ process $\{X_t\}$ will have singular points for $u=0$ and $\omega_1 = (4n\pm1)\frac{\pi}{m}$ for all $m$ and $-\frac{m\pm1}{4}<n<\frac{m\mp1}{4}.$\\
This proves the part $(a)$.

%Now, at $\omega_1, (mu - (2\cos\left(\frac{\omega_1 m}{2}\right)\cos\left(\frac{(2-m)\omega_1)}{2}\right))^2 = 0$ only when \\
%\vspace{1em}

\item Again the term $$4\cos^{2}\left(\frac{m\omega}{2}\right)\sin^{2}\left(\frac{(m-2)\omega}{2}\right) = 0$$ when $$\sin^2\left(\frac{(m-2)\omega}{2}\right) = 0$$
$$\sin\left(\frac{(m-2)\omega}{2}\right) = 0 \text{ for }
\omega_2 = \frac{\pm 2n\pi}{m-2} \text{ for all } m\in {\mathbb{N-} \{2\}},\\ \text{ and } n=0, \pm1, \pm2, \cdots.$$
At $\omega_2$, the second term will become zero if and only if,
\begin{align*}
\left[mu - \left(2\cos\left(\frac{m\omega_2}{2}\right)\cos\left(\frac{(2-m)\omega_2}{2}\right)\right)\right]^2 &= 0\\
%mu - \left[2\cos\left(\frac{\pm mn\pi}{m-2}\right)\cos\left(\pm n\pi\right)\right] &=0\\
%mu - 2(-1)^n\cos\left(2n\pi\pm\frac{4n\pi}{m-2}\right) &=0\\
\Rightarrow mu - 2(-1)^n\cos\left(\pm\frac{4n\pi}{m-2}\right) &=0\\
\Rightarrow u = \frac{2(-1)^n}{m}\cos\left(\frac{4n\pi}{m-2}\right).
\end{align*}
This proves the part $(b)$.

\item \noindent In \eqref{eq5_} let $U(\omega)=(2+m^2u^2-2mu(\cos(\omega)+\cos((1-m)\omega))+2\cos(m\omega))$. For different values of $m=1,2,3,4$ and $0\leq u<2/m$ in figure \ref{fig_1}, we observe that the function $U(\omega)$ does not attain $0$ for $\omega\in(-\pi,\pi)$. This signifies that the spectral density defined by \eqref{eq5_} has no singularity for $m=1,3,4$ and $0\leq u<2/m$. For $m=2$, the spectral density is unbounded since $U(\omega)$ takes value $0$ at $\omega=\cos^{-1}(u)$. Therefore, we conclude that for $m=2$ the singularities are at $\omega = \cos^{-1}(u)$.
\end{enumerate}
\end{proof}
\noindent In figure \ref{fig_1}, observe the behaviour of function $U(\omega)$ for $\omega \in (-\pi,\pi)$ and for different values of $m$ and $u$. For $m=2$, the function $U$ touches the $x$-axis for all values of $u$. Further, for $m=1$ it touches the $x$-axis only for $u=0$. For other cases see Theorem \ref{theorem4}.
\begin{figure}[ht!]
\centering
{\includegraphics[width=11cm,height=9cm]{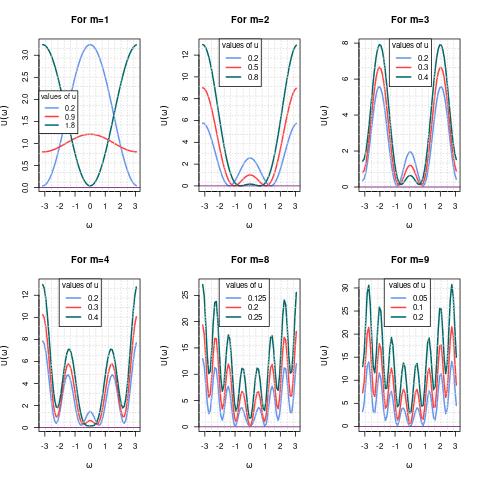} }%
\caption{Plot of function $U(\omega)$ for different values of $m \in\{1, 2, 3, 4, 8, 9\}$ and $0\leq u\leq 2/m$.}%
\label{fig_1}%
\end{figure}

\begin{definition}[Slowly varying function \cite{Gray1989}]
A function $b(\omega)$ is said to be slowly varying at $\omega_0$ if for $\delta > 0$, $(\omega-\omega_0)^\delta b(\omega)$ is increasing and $(\omega-\omega_0)^{-\delta} b(\omega)$ is decreasing in some right-hand neighborhood of $\omega_0$. Also, $(\omega-\omega_0)^\delta b(\omega)$ is decreasing and $(\omega-\omega_0)^{-\delta} b(\omega)$ is increasing in some left-hand neighbourhood of $\omega_0$.
\end{definition}
\noindent We need the following lemma which is given in \cite{Gray1989} to prove our next result.
%%%%%%%%%%%%%%%%%%%%%%%%%%%%%%%%%%%%%%%%%%%%%%%%%%%%%%%%%%%%%%%%%%%
\begin{lemma}[Gray et al.\cite{Gray1989}]\label{lemma1}
Let $R(\tau)=\int_{0}^{\pi} P(\omega) d\omega$ where $\tau$ is an integer and $P(\omega)$ is spectral density. Suppose $P(\omega)$ can be expressed as
\begin{align}
P(\omega)=b(\omega)\lvert\omega-\omega_0\rvert^{-\beta} \label{lemma_1}
\end{align}
with $0<\beta<\frac{1}{2}$ and $\omega_0\in (0,\pi)$. Further, suppose that $b(w)$ is non-negative and of bounded variation in $(0, \omega_0-\epsilon)\cup
(\omega_0+\epsilon, \pi)$ for $\epsilon>0$. Also suppose that $b(\omega)$ is slowly varying at $\omega_0$, then as $\tau\to \infty$
\begin{align*}
    R(\tau)\simeq\tau^{2\beta-1}cos(\tau\omega_0).
\end{align*}
\end{lemma}
%%%%%%%%%%%%%%%%%%%%%%%%%%%%%%%%%%%%%%%%%%%%%%%%%%%%%%%%%%%%%%%%%%
\begin{theorem}\label{theorem1}
The stationary type $1$ HARMA$(p,\nu,0,q)$ process has seasonal long memory for $0<\nu<1/2$. 
\end{theorem}

\begin{proof}
The spectral density of type $1$ HARMA$(0,\nu,u,0)$ process is given by
\textbf{}\begin{align*}
f_x(\omega) &= \dfrac{\sigma^2}{2\pi}(2+m^2u^2-2mu(\cos(\omega)+\cos((1-m)\omega))+2\cos(m\omega))^{-\nu}.
\end{align*}
For $u=0$ the spectral density has the form
\begin{align}\label{spec_1}
f_x(\omega)= \dfrac{\sigma^2}{2\pi}(2+2\cos(m\omega)).
\end{align}
Also, the spectral density is unbounded at $\omega_0=(4n\pm 1)\frac{\pi}{m}, -\frac{m\pm1}{4}<n<\frac{m\mp1}{4}.$ which implies that the covariance is not absolutely summable for $u=0$ at frequency $\omega_0$. To prove the process is seasonal long memory we use lemma \ref{lemma1} defined by Gray et al. \cite{Gray1989}. 
Now \eqref{spec_1} can be rewritten as
\begin{align*}
f_x(\omega)= \frac{\sigma^2 (2+2\cos(m\omega))^{-\nu}\lvert\omega-\omega_0\rvert^{-2\nu}}{2\pi\lvert\omega-\omega_0\rvert^{-2\nu}}.
\end{align*}
Comparing the above equation with \eqref{lemma_1}
\begin{align*}
b(\omega)=\frac{\sigma^2 (2\cos(m\omega)+2)^{-\nu}}{2\pi\lvert\omega-\omega_0\rvert^{-2\nu}}.
\end{align*}

\noindent Now to show $b(\omega)$ is slowly varying at $\omega_0$, consider the case $\omega>\omega_0$ and for $\delta>0$ define, 
\begin{align*}
l(\omega)=b(\omega)(\omega-\omega_0)^\delta
&=\frac{\sigma^2}{2\pi} (2+2\cos(m\omega))^{-\nu}(\omega-\omega_0)^{\delta+2\nu}
\end{align*}
and
\begin{align*}
l'(\omega)&=\frac{\sigma^2}{2\pi}(\omega-\omega_0)^{\delta+2\nu-1}(2+2\cos(m\omega))^{-\nu-1}((\delta+2\nu)(2+2\cos(m\omega)\\
&+2\nu m\sin(m\omega)(\omega-\omega_0)).
\end{align*}
For $\omega>\omega_0$ the terms $(\omega-\omega_0)^{\delta+2\nu-1}$, $(\delta+2\nu)$ and $(2+2\cos(mw))$ are positive. It can be easily shown that 
\begin{align*}
\lim_{\omega\to\omega_0} (2\nu m\sin(m\omega)(\omega-\omega_0)+(\delta+2\nu)(2+2\cos(m\omega)))>0.
\end{align*}

\noindent Thus in some right hand neighbourhood of $\omega_0$, i.e. for $\omega \to \omega_0^{+}$, $l'(\omega)> 0$ and $(\omega-\omega_0)^{\delta}b(\omega)$ is increasing and similarly $(\omega-\omega_0)^{-\delta}b(\omega)$ is decreasing when $\omega\to\omega_0^{+}$. Similarly, it can be easily shown that for $\omega<\omega_0$, $(\omega-\omega_0)^{\delta}b(\omega)$ is decreasing and $(\omega-\omega_0)^{-\delta}b(\omega)$ is increasing in some left hand neighbourhood of $\omega_0$. Thus the function is slowly varying at $\omega_0$.\\
Also, it can be easily verified that the function $b(w)$ has bounded derivative in $(0,\omega_0-\epsilon)\cup(\omega_0+\epsilon, \pi),$  hence it is of bounded variation in $(0,\omega_0-\epsilon)\cup(\omega_0+\epsilon, \pi).$ 

\noindent Using the above two results and the lemma \ref{lemma1} the autocorrelation function $R(h)$ of the type 1 Humbert ARMA process takes the following asymptotic form
\begin{align}\label{new}
R(h)\simeq h^{2\nu-1}\cos(h\omega_0),\;\mbox{as}\;h\rightarrow \infty.
\end{align}
\noindent The result \eqref{new} implies that the process is seasonal long memory for $0<\nu<1/2$.

\end{proof}

%%%%%%%%%%%%%%%%%%%%%%%%%%%%%%%%%%%%%%%%%%%%%%%%%
\subsection{Pincherle ARMA $(p, \nu, u, q)$ Process}
This section deals with the special case of the type 1 HARMA process for $m=3$. The Pincherle polynomials are polynomials introduced by Pincherle (1891) \cite{Pincherle1891}. The Pincherle polynomials were generalized to Humbert Polynomials by Humbert (1920) \cite{Humbert1920}. 
\begin{definition}[Pincherle polynomials]
The Pincherle polynomials $P_n^\nu(u)$ are defined as the coefficient of $t$ in the expansion of $(1-3ut+t^n)^{-\nu}$. The Pincherle polynomials are defined by taking $m=3$ in type 1 Humbert polynomials that is $P_n^\nu(u)=\Pi_{n,3}^\nu(u)$. Also, the generating function relation for Pincherle polynomials is given by
\begin{align*}
(1-3ut+t^n)^{-\nu}=\sum_{n=0}^{\infty}P_n(u) t^n, \;\lvert t\rvert<1, \lvert\nu\rvert<1/2, \; \lvert u\rvert\leq 1.
\end{align*}
\end{definition}
\noindent The polynomials satisfies the following difference equation \cite{Baker1921}
\begin{align*}
    (n+1)P_{n+1}^\nu(u)-3u(n+\nu)P_n^\nu(u)u+(n+3\nu-2)P_{n-2}^\nu(u)=0.
\end{align*}
\noindent The coefficient of Pincherle polynomials can be written as $P_0^\nu(u)=1$, $P_1^\nu(u)=3\nu, P_2^\nu(u)=9\nu(\nu+1)u^2/2$ and the $nth$ coefficient takes the form \cite{Baker1921} 

\[\dfrac{\Gamma(n+\nu)\Gamma(1/3)\Gamma(2/3)}{\Gamma(\nu)\Gamma((n+1)/3)\Gamma((n+2)/3)\Gamma((n+3)/3)}.\]
\begin{definition}[Pincherle ARMA process]
\noindent The Pincherle ARMA $(p, \nu, u, q)$ process is defined by taking $m=3$ in type 1 HARMA process defined in \eqref{eq1_} and the process has the form defined below
\begin{align}\label{pin_arma}
\Phi(B) (1-3uB+ B^3)^\nu X_t= \Theta(B) \epsilon_t,
\end{align}
where $\epsilon_t$ is Gaussian white noise with variance $\sigma^2$, $0\leq u<2/3$ and $B$, $\Phi(B)$ and $\Theta(B)$ are lag, stationary AR and invertible MA operators respectively defined in definition \ref{def1}.
\end{definition}

\begin{theorem}\label{theorem5}
Let $\{X_t\}$ be the Pincherle \textup{ARMA}$(p,\nu,u,q)$ process defined in \eqref{pin_arma} and all roots of $\Phi(B)=0$ and $\Theta(B)=0$ lie outside the unit circle then the Pincherle \textup{ARMA}$(p,\nu,u,q)$ process is stationary and invertible for $\lvert\nu\rvert<1/2$ and $0\leq u\leq2/3.$ 
\end{theorem}
\begin{proof}
The proof can be easily done by taking $m=3$ in the proof of theorem \ref{theorem3}.
\end{proof}

\begin{theorem}\label{theorem6}
The stationary Pincherle HARMA$(p,\nu,0,q)$ process has seasonal long memory for $0<\nu<1/2$ at $\omega_0=\pi/3$. 
\end{theorem}
\begin{proof}
    According to theorem \ref{theorem4} the spectral density of Pincherle \textup{ARMA} process has sigularity at $u=0$ for $\omega_0=\pi/3$. Also, similar to the proof of theorem \ref{theorem1} the autocovariance function of Pincherle \textup{ARMA} process $\gamma(h)$ has the asymptotic form $R(h)\simeq h^{2\nu-1}\cos(h\omega_0)$. This proves that the process has seasonal long memory for $0<\nu<1/2$ at $\omega_0=\pi/3$.
\end{proof}

\begin{theorem}
For a Pincherle ARMA$(p, \nu, u, q)$ process defined in \eqref{pin_arma}, the spectral density takes the following form 
\begin{align*}
f_x(\omega)&=\dfrac{\sigma^2}{2\pi}\dfrac{\lvert\Theta(z)\rvert^2}{\lvert\Phi(z)\rvert^2}(8\cos^3(\omega)-12u\cos^2(\omega)-C\cos(\omega)+D)^{-\nu},
\end{align*}
where $z=e^{-\iota \omega},\; C = 6+6u, \text{ and }D = 2+6u+9u^2$.
% \[\psi(l)= \sum_{s=\min[0,l]}^{\max[q,q+l]}\theta_s\theta_{s-l}\]
% and \[\zeta_j=\dfrac{\sigma^2}{2\pi}\Bigg[\rho_j\prod_{i=1}^{p}(1-\rho_i\rho_j)\prod(\rho_j-\rho_m)\Bigg]^{-1}.\]
\end{theorem}

\begin{proof}

Taking $m=3$ in \eqref{eq5_} gives us the desired spectral density. 

\end{proof}

\begin{theorem}
The autocovariance function for Pincherle ARMA process takes the following form
\begin{align*}
\gamma(h)=\sigma^2\displaystyle\sum_{j=0}^{\infty}\displaystyle\sum_{n=0}^{\infty}\psi_j\psi_{j+h}P_n^\nu(u)P_{n+h}^\nu(u).
\end{align*}
\end{theorem}

\begin{proof}
For lag $h$ the autocovariance of the process $\{X_t\}$ and $\{X_{t+h}\}$ using the  \eqref{pin_arma} is given by
\begin{align*}
\mathrm{Cov}(X_t X_{t+h})=\mathrm{E}[X_t X_{t+h}],
\end{align*}
where $X_t$ can be written as
\begin{align*}
X_t=\displaystyle\sum_{j=0}^{\infty}\displaystyle\sum_{n=0}^{\infty}\psi_jP_n^\nu(u)\epsilon_{t-j-n}
\end{align*}
and 
\begin{align*}
\mathrm{E}[X_t X_{t+h}]=\sigma^2\sum_{j=0}^{\infty}\sum_{n=0}^{\infty}\psi_j\psi_{j+h} P_n^\nu(u)P_{n+h}^\nu(u).
\end{align*}
\end{proof}
%%%%%%%%%%%%%%%%%%%%%%%%%%%%%%%%%%%%%%%%%%%%%%

\section{Type 2 HARMA$(p, \nu, u, q)$ Process}\label{sec3}
\setcounter{equation}{0}
Milovovic and Dordevic (1987)\cite{Milovanovic1987} considered the following generalization of Gegenbaur polynomials, which we call type 2 Humbert polynomials and are used to define type 2 HARMA process.
\begin{definition}[Type 2 Humbert polynomials]
The type 2 Humbert polynomials are defined by considering the polynomials $Q_{n,m}^\nu(u)$ defined by the following generating function  
\begin{align}\label{eq_7}
(1-2ut+t^m)^{-\nu}=\displaystyle\sum_{n=0}^{\infty}Q_{n,m}^\nu(u) t^n,\;\lvert t\rvert<1,\;\lvert\nu\rvert<1/2,\;\lvert u\rvert\leq 1.
\end{align}
\noindent Here  $Q_{n,m}^\nu(u)=\Pi_{n,m}^\nu(\frac{2u}{m})$ $(\text{see } \eqref{hum})$.
\end{definition}
\noindent The explicit form of the polynomials $Q_{n,m}^\nu(u)$ is defined by 
\begin{align*}
Q_{n,m}^\nu(u)=\sum_{k=0}^{[\frac{n}{m}]}(-1)^{k}\frac{(\nu)_{(n-(m-1)k)}}{k!(n-mk)!}(2u)^{n-mk},
\end{align*}
where $\nu_0=1$ and $(\nu)_n=\nu(\nu+1)\hdots(\nu+n-1)$.
\begin{definition}[Type 2 HARMA process]
The type 2 \textup{HARMA} process is defined by using the above defined generation function as follows
\begin{align}\label{eq_6}
\Phi(B) (1-2uB+ B^m)^\nu X_t= \Theta(B) \epsilon_t,
\end{align} 
where $\epsilon_t$ is Gaussian white noise with variance $\sigma^2$, $0\leq u<1$, and $B$, $\Phi(B)$, $\Theta(B)$ are lag, stationary AR and invertible MA operators respectively defined in definition \ref{def1}.  
\end{definition}
\noindent For $m=2$ the above polynomials in \eqref{eq_7} is Gegenbauer polynomials and $Q_{n,2}^\nu(u)=C_n^\nu(u)$. Also, for $m=3$ the polynomials in \eqref{eq_7} are known as Horadam-Pethe polynomials and for $m=1$ they are known as Horadam polynomials, see Gould (1965) \cite{Gould1965}, Horadam (1985) \cite{Horadam1985} and Horadam and Pethe (1981) \cite{Horadam1981}. 

\begin{theorem}\label{theorem8}
Let $\{X_t\}$ be the type 2 \textup{HARMA}$(p,\nu,u,q)$ process and all roots of $\Phi(B)=0$ and $\Theta(B)=0$ lies outside the unit circle then the \textup{HARMA}$(p,\nu,u,q)$ process is stationary and invertible for $\lvert\nu\rvert<1/2$ and $0\leq u\leq 1.$ 
\end{theorem}
\begin{proof}
The process is stationary and invertible for $\lvert\nu\rvert<1/2$ and $0\leq u\leq1$ can be easily proved using the proof for the stationarity of type 1 \textup{HARMA} process defined in \ref{theorem3}.
\end{proof}

\begin{theorem}
For a type 2 Humbert ARMA$(p, \nu, u, q)$ process defined in \eqref{eq_6}, under the assumptions of theorem \ref{theorem8} the spectral density takes the following form
\begin{align}\label{SD1_}
f_x(\omega)&=\dfrac{\sigma^2}{2\pi}\dfrac{\lvert\Theta(z)\rvert^2}{\lvert\Phi(z)\rvert^2}(2+4u^2-4u(\cos(\omega)+\cos((1-m)\omega))+2\cos(m\omega))^{-\nu},
\end{align}
where $z=e^{-\iota \omega}$.
\end{theorem}
\begin{proof}
\noindent Rewrite \eqref{eq_6} as follows
\begin{equation*}
X_t=\Psi(B)\epsilon_t,
\end{equation*}
where $\Psi(B)= \dfrac{\Theta(B)}{\Phi(B)} \Delta^{\nu}$ and $\Delta^{\nu}=(1-2uz+z^m)^{-\nu}$.  The spectral density of the innovation process $\epsilon_t$ is given by $\sigma^2/2\pi$, which implies

\begin{align}
f_x(\omega)=\dfrac{\sigma^2}{2\pi}\lvert\Psi(z)\rvert^2
=\dfrac{\sigma^2}{2\pi}\dfrac{\lvert\Theta(z)\rvert^2}{\lvert\Phi(z)\rvert^2}\left\lvert1-2uz+z^{m}\right\rvert^{-2\nu},
\end{align}
where $z=e^{-\iota \omega}$. Furthermore,
\begin{align*}
\left\lvert1-2ue^{-\iota\omega}+e^{-m\iota\omega}\right\rvert^{-2\nu} = (2+4u^2-4u(\cos(\omega)+\cos((1-m)\omega))+2\cos(m\omega))^{-\nu},
\end{align*}
and the spectral density takes the following form

\begin{align*}
f_x(\omega) &= \dfrac{\sigma^2}{2\pi}\dfrac{\lvert\Theta(z)\rvert^2}{\lvert\Phi(z)\rvert^2}(2+4u^2-4u(\cos(\omega)+\cos((1-m)\omega))+2\cos(m\omega))^{-\nu}.
\end{align*}
\end{proof}

\begin{theorem}\label{theorem7}
Under the assumption of theorem \ref{theorem8} let $\{X_t\}$ be the type 2 \textup{HARMA}$(p, \nu, u, q)$ process then the spectral density of \textup{HARMA}$(p, \nu, u, q)$ process has singularities
\begin{enumerate}[\itshape(a)]
\item at $u=0$ and $\omega = \frac{4n\pi\pm\pi}{m}$ for $-\frac{m\pm1}{4}<n<\frac{m\mp1}{4}.$
\item at $u = (-1)^n\cos(\frac{4n\pi}{m-2})$ and $\omega = \pm \frac{2n\pi}{m-2}$ for $m \neq 2$ and $-\frac{(m-2)}{4}< n<\frac{(m-2)}{4}$.
% \item at $\omega = \cos^{-1}(u)$ for $m=2$.
%the spectral density $f_x(\omega) = \dfrac{\sigma^2}{2\pi}\dfrac{\lvert\Theta(z)\rvert^2}{\lvert\Phi(z)\rvert^2}(2+m^2u^2+2\cos(m\omega)-2mu(\cos(\omega)+\cos((m-1)\omega)))^{-\nu}$ has . 
\end{enumerate}
\end{theorem}

\begin{proof}
The spectral density of type 2 HARMA process is 
\begin{align}
    f_x(\omega) &= \dfrac{\sigma^2}{2\pi}\dfrac{\lvert\Theta(z)\rvert^2}{\lvert\Phi(z)\rvert^2}(2+4u^2-4u(\cos(\omega)+\cos((1-m)\omega))+2\cos(m\omega))^{-\nu}.
\end{align}
Similar to the proof in Theorem \ref{theorem4}, we find the zeros by writing the denominator as follows
\begin{align}\label{eq_8}
   &2+4u^2-4u(\cos(\omega)+\cos((1-m)\omega))+2\cos(m\omega)=\nonumber\\&4\cos^{2}\left[\frac{m\omega}{2}\right]\sin^{2}\left[\frac{(m-2)\omega}{2}\right] + \left[2u - 2\cos\left(\frac{m\omega}{2}\right)\cos\left(\frac{(2-m)\omega}{2}\right)\right]^2 
\end{align}
The proof of part {\em (a)} is same as to the part {\em (a)} of theorem \ref{theorem4}. To prove the part {\em (b)} the term 
$4\cos^{2}\left[\frac{m\omega}{2}\right]\sin^{2}\left[\frac{(m-2)\omega}{2}\right]=0$ at $\omega_0=\frac{\pm 2n\pi}{m-2}$. For this $\omega_0$ the second term of \eqref{eq_8} is zero for $u=(-1)^n\cos(\frac{4n\pi}{m-2})$ for all $m\in {\mathbb{N-} \{2\}}$ and $-\frac{(m-2)}{4}< n<\frac{(m-2)}{4}$.

\end{proof}
\noindent The particular cases of type 2 Horadam ARMA process is discussed as follws:

\subsection{Horadam ARMA$(p, \nu, u, q)$ Process }
\begin{definition}[Horadam polynomials]
In \eqref{eq_7} by taking $m=1$ the reduced polynomials are known as Horadam polynomials. The Horadam polynomials are defined as the coefficient of t in the expansion of $(1-2ut+t)$ and the generating function relation is given as follows
\begin{align*}
(1-2ut+t)^{-\nu}=\displaystyle\sum_{n=0}^{\infty}Q_{n,1}^\nu(u) t^n,\:\lvert t\rvert<1,\;\lvert\nu\rvert<1/2,\;\lvert u\rvert\leq 1.
\end{align*}
\end{definition}

\begin{definition}[The Horadam ARMA process]
The time series process defined using the generating function of Horadam polynomials are defined by Horadam \textup{ARMA} process, which is a special case of type2 HARMA process for m=1 and the process takes the following form
\begin{align}\label{Horadam}
\Phi(B) (1-2uB+ B)^\nu X_t= \Theta(B) \epsilon_t,
\end{align}
\end{definition}
\noindent where $\epsilon_t$ is Gaussian white noise with variance $\sigma^2$, $0\leq u<1$, and $B$, $\Phi(B)$, $\Theta(B)$ are lag, stationary AR and invertible MA operators respectively defined in definition \ref{def1}.
\begin{theorem}
For a Horadam ARMA$(p, \nu, u, q)$ process defined in \eqref{Horadam}, the spectral density takes the following form
\begin{align*}
f_x(\omega)&=\dfrac{\sigma^2}{2\pi}\dfrac{\lvert\Theta(z)\rvert^2}{\lvert\Phi(z)\rvert^2}(2+4u^2-4u-4u\cos(\omega)+2\cos(\omega))^{-\nu},\;z=e^{-\iota \omega}.
\end{align*}

\end{theorem}
\begin{proof}
This can be easily proved by taking $m=1$ in the spectral density of type 2 \textup{HARMA} process defined in \eqref{SD1_}.
\end{proof}

\subsection{Horadam-Pethe ARMA$(p, \nu, u, q)$ Process}
Taking $m=3$ in \eqref{eq_7} the reduced form of the polynomials is known as Horadam-Pethe polynomials and the corresponding time series defined using the generating function of Horadam-Pethe polynomials is known as Horadam-Pethe ARMA process defined as follows
\begin{align}\label{HoradamP}
\Phi(B) (1-2uB+ B^3)^\nu X_t= \Theta(B) \epsilon_t,
\end{align}
where $(1-2uB+ B^3)^{-\nu}=\sum_{n=0}^{\infty}Q_{n,3}^\nu(u) t^n.$

\begin{theorem}
Under the assumptions of theorem \ref{theorem8} for a Horadam-Pethe ARMA$(p, \nu, u, q)$ process defined in \eqref{HoradamP}, the spectral density takes the following form
\begin{align*}
f_x(\omega)&=\dfrac{\sigma^2}{2\pi}\dfrac{\lvert\Theta(z)\rvert^2}{\lvert\Phi(z)\rvert^2}(2+4u^2-4u(\cos(\omega)+\cos(2\omega))+2\cos(3\omega))^{-\nu},
\end{align*}
where $z=e^{-\iota\omega}$.
\end{theorem}

\begin{remark}
Taking $m=2$ the polynomials in \eqref{eq_7} reduced to Gegenbauer polynomials and $Q_{n,2}^\nu(u)=~C_n^{\nu}(u)$. Moreover, the corresponding time series using the generating function of Gegenbauer polynomials namely Gegenbauer Autoregressive Moving Average (GARMA) process is studied by Gray and Zhand in 1989 (see \cite{Gray1989}). 
\end{remark}

\begin{remark}
The stationarity and invertibility condition for Horadam ARMA and Horadam-Pethe ARMA process is same as the type 2 HARMA process, which is the process is stationary and invertible if all roots of $\Phi(B)=0$ and $\Theta(B)=0$ lies outside the unit circle and $\lvert\nu\rvert<1/2$ and $0\leq u<1.$ 
\end{remark}
\noindent The time-series plots for simulated Pincherle, Horadam, Horadam-Pethe and Gegenbauer ARMA processes are given in the Figure \ref{fig:new_}. We simulated time-series of size 1000 from each processes. All these series have in theory infinite differencing terms. We consider only finite terms by truncating the binomial expansions of the different shift operators. For Pincherle ARMA process the relation defined in \eqref{pin_arma} is used, that is
\begin{align}\label{sim1}
X_t=\frac{ \Theta(B)}{\Phi(B) } (1-3uB+ B^3)^{-\nu} \epsilon_t.
\end{align}
The series $Z_t=(1-3uB+ B^3)^{-\nu} \epsilon_t$ is generated using the simulated innovation series $\epsilon_t\sim\mathcal{N}(0,\sigma^2)$. Further, we approximate $Z_t$ by considering first $4$ terms in the binomial expansion of $(1-3uB+ B^3)^{-\nu}$, which is
\begin{align}
Z_t=(1-3uB+ B^3)^{-\nu} \epsilon_t&=\sum_{n=0}^{\infty} \sum_{j=0}^{n}(-1)^j{\frac{(\nu)_n}{n!}} \binom{n}{j} (3u)^{n-j}B^{2j+n} \epsilon_{t}\\
&\approx\sum_{n=0}^{4} \sum_{j=0}^{n}(-1)^j{\frac{(\nu)_n}{n!}} \binom{n}{j} (3u)^{n-j}\epsilon_{t-n-2j}.
\end{align}
Now by generating the series $Z_t$ the equation \eqref{sim1} takes the following form
\begin{align*}
X_t=\frac{ \Theta(B)}{\Phi(B)} Z_t,
\end{align*}
which is nothing but the ARMA process which is simulated using the inbuilt R library by passing the $Z_t$ as innovation series. Using the same approach, we simulate the Hordam, Gegenbauer and Horadam Pethe ARMA processes by taking the binomial expansion of $(1-2uB+ B^m)^{-\nu}$, for $m=1,2$ and $3$ respectively.

\begin{figure}[ht!]
\centering
{\includegraphics[width=9cm,height=11cm]{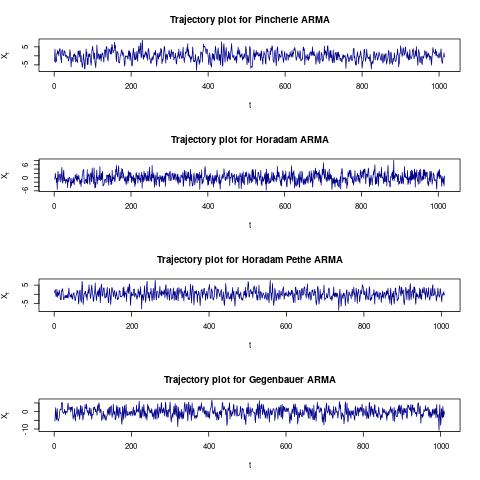} }%
\caption{Trajectory plots for Pincherle, Horadama, Horadam-Pethe and Gegenbauer ARMA processes for $p=1, q =0, \nu=0.3$ and $u=0.1$.}%
\label{fig:new_}%
\end{figure}

\section{Conclusions}
We study the general Humbert polynomials based autoregressive moving average called here HARMA $(p, \nu, u, q)$ time series models. Initially, type 1 HARMA $(p, \nu, u, q)$ process defined in \eqref{eq1_} and it's stationarity and invertibility conditions are derived. We also compute the spectral density of the above process. For $m=3$ in \eqref{pin_arma}, we focus on particular case Pincherle ARMA $(p, \nu, u, q)$ process, by obtaining the spectral density and also prove that for $u=0$ and $0<\nu<1/2$, the process also exhibits seasonal long memory property. In subsequent section, we study similar properties of particular cases of type 2 HARMA $(p, \nu, u, q)$ process defined in \eqref{eq_6} for $m=1$ and $m=3$ named as Horadam ARMA process and Horadam-Pethe ARMA process respectively.

\noindent Further, we believe that the proposed time series models will be helpful in modelling of real world data. In future the application of these models will be discovered. Also, the  estimation techniques for example minimum contrast estimation \cite{Anh2004, Anh2007} will be applied for the discussed models. This technique estimates the parameters by minimizing the spectral density and empirical spectral density of the process. Maximum likelihood estimation is the particular case of minimum contrast estimation. Apart from this, Pincherle, Horadam and Horadam-Pethe random fields will be interest of study on the line of Gegenbauer random fields \cite{Espejo2014}. Moreover, one can study the tempered versions of Humbert, Pincherle, Horadam and Horadam-Pethe ARMA processes similar to Sabzikar et al. \cite{Sabzikar2019}.\\

\noindent {\bf Acknowledgements:} 
Nikolai Leonenko (NL) would like to thank the Isaac Newton Institute for Mathematical Sciences for support and hospitality during the programme Fractional Differential Equations (FDE2). Also NL was partially supported under the Australian Research Council's Discovery Projects funding scheme (project number  DP220101680), LMS grant 42997 (UK) and FAPESP (Brazil) grant. Further, Niharika Bhootna and Monika S. Dhull would like to thank Ministry of Education (MoE), India for supporting their PhD research.\\

%%%%%%%%%%%%%%%%%%%%%%%%%%%%%%%%%%%%%%%%%%%%%%%

\end{document}